\documentclass[12pt]{amsart}
\usepackage{amsmath,amssymb,amsfonts}
\usepackage[colorlinks=true,urlcolor=green,linkcolor=red,citecolor=blue]{hyperref}
\usepackage{geometry} 

\newtheorem{theorem}{Theorem}[section]
\newtheorem{proposition}{Proposition}[section]
\newtheorem{cor}{Corollary}[section]
\newtheorem{lemma}{Lemma}[section]
\newtheorem{remark}{Remark}[section]

\newtheorem{definition}{Definition}[section]
\renewcommand {\epsilon}{\varepsilon}

\newcommand{\N}{\mathbb{N}}

\newcommand{\R}{\mathbb{R}}

\newcommand{\ioe}{\leq}

\begin{document}
\title{On the class of order L-weakly	and order M-weakly
compact operators}
\author[D. Lhaimer]{Driss Lhaimer}
\address{Department
of Mathematics, Faculty of Sciences, Ibn Tofail University, P.O. Box 133, Kenitra 14000, Morocco}
\email{driss.lhaimer@uit.ac.ma}
\email{mohammed.moussa@uit.ac.ma}
\author[K. Bouras]{khalid Bouras}
\address{Department
of Mathematics, Faculty Polydisciplinary, Abdelmalek Essaadi University, P.O. Box 745, Larache 92004, Morocco}
\email{bouraskhalid@hotmail.com}
\author[M.Moussa]{Mohammed Moussa}
\address{Department
of Mathematics, Faculty of Sciences, Ibn Tofail University, P.O. Box 133, Kenitra 14000, Morocco}
\begin{abstract}
In this paper, we introduce and study new concepts of order L-weakly and order M-weakly compact operators. As consequences, we obtain some characterizations of Banach lattices
with order continuous norms or whose topological duals have order continuous norms.\\
It is proved that if $T : E \longrightarrow F$ is an operator between two Banach lattices, then $T$ is order M-weakly compact if and only if its adjoint $T'$ is order L-weakly compact. Also, we show that if its adjoint $T'$ is order M-weakly compact, then $T$ is order L-weakly compact. Some related results are also obtained.
\end{abstract}

\maketitle
\section{Introduction and notation}
Throughout this paper $X$ and $Y$ will denote real Banach spaces, $E$ and $F$ will denote real Banach lattices and $B_ X$ will denote the closed unit ball of $X$. 
We will use the term operator, between two Banach spaces, to mean a bounded linear mapping. The space of all operators from $X$ into $Y$ will be denoted by $L(X, Y )$.
For the convenience of the reader, let us recall some notions and results that this work involves.\\
$E^a$ denotes the maximal ideal in $E$ on which the induced norm is order continuous. Note that $E^a$ is closed and
$$E^ a = \{x \in E ~:~ each~ monotone ~sequence ~in~ [0, |x|]~ is~ convergent\}.$$ 
The difficulty of studying weakly compact operators in Banach lattices resulted in introduction of related notions of L-weak compactness and M-weak compactness.
A nonempty bounded subset $A$ of $E$ is called L-weakly compact if for every disjoint sequence $(x_n)$ in the
solid hull of $A$ we have $\lim \Vert x_n \Vert=0$. Note that every L-weakly compact subset $A \subset E$ is relatively weakly compact \cite[$Proposition~3.6.5$]{MN1}.
The classes of L-weakly and M-weakly compact operators were introduced by Meyer-Nieberg in \cite{MN2}. An operator $T$ from $X$ into $F$ is called L-weakly compact if $T (B_ X )$ is an L-weakly compact subset of $F$. An operator $T$ from $E$ into $Y$ is called M-weakly compact if $\lim  T (x_ n ) = 0$ holds for every norm bounded disjoint sequence $(x_ n )$ in $E$.\\
Following P. G. Dodds \cite{DF}, we shall say that an operator $T : E \longrightarrow Y$ is order weakly compact whenever $T [0, x]$ is a relatively weakly compact subset of $Y$ for each $x \in E^+$.\\
We introduce a new class of order L-weakly (resp. M-weakly) compact operators. An operator $T$ from a Banach lattice $E$ into a Banach lattice $F$ is called order L-weakly compact whenever $T[0,x]$ is an L-weakly compact subset of $F$ for each $x\in E^+$, and an operator $T$ from
a Banach lattice $E$ into a Banach lattice $F$ is called order M-weakly compact if for every disjoint sequence $(x_n)$ in $B_E$ and every  order bounded  sequence $(f_ n ) $ of $F'$ we have $f_n(T(x_n))\rightarrow 0$. The space of order L-weakly (resp. M-weakly) compact operators between $E$ and $F$ will be denoted by $oLW(E,F)$ (resp. $oMW(E,F)$).\\
Note that the class of order L-weakly (resp. M-weakly) compact operators contains strictly that of L-weakly (resp. M-weakly) compact operators.
On the other hand, it is easy to see that each order L-weakly compact operator is order weakly compact, but the converse is false in general. 
We begin by establishing sequential characterizations of order
L-weakly compact operators. As consequences, we will give some interesting results.\\
We know that the classes of L-weakly and M-weakly compact operators are in duality with each other (an operator $T$ between two Banach
lattices is L-weakly compact (resp. M-weakly compact) if and only if its adjoint $T'$ is M-weakly compact (resp. L-weakly compact)) \cite[$Proposition~3.6.11$]{MN1}. As we shall see, a similar result for the classes of order L-weakly and order M-weakly compact operators will be proved.\\
An operator $T~:E\longrightarrow F$ is called order bounded if it maps order bounded subsets of $E$ to order bounded subsets of $F$. Note that every positive operator is order bounded. We introduce a new class of strong order bounded operators. An operator $T$ from $X$ into $F$ is called strong order bounded if it maps the unit ball of $X$ into an order bounded subset of $F$. Note that every strong order bounded operator between two Banach lattices is order bounded. We conclude this paper by giving necessary and sufficient condition under which each order (resp. strong order) bounded operator is order L-weakly compact, order M-weakly compact (resp. L-weakly compact, M-weakly compact).\\
For the theory of Banach lattices and operators, we refer the reader to the monographs \cite{AB1,MN1,Z}.
\section{MAIN RESULTS}
We start by the following definitions.
\begin{definition}
An operator $T$ from $E$ into $F$ is called order
L-weakly compact whenever $T[0,x]$ is an L-weakly compact subset of $F$ for each $x\in E^+$.
\end{definition}
\begin{remark} ${}$
\begin{itemize}
\item Note that 
$$ \begin{array}{l l} 
Id_E $ is order L-weakly compact$ &  \Longleftrightarrow \forall x\in E^+,~[0,x] $ is L-weakly compact$.  \\
& \Longleftrightarrow \forall x\in E^+, x\in E^a.\\ &\Longleftrightarrow E^a=E.\\ & \Longleftrightarrow$ The norm of $ E $ is order
continuous$.
\end{array}$$
\item Clearly, every L-weakly compact operator is order
L-weakly compact (it suffices to note that every order interval of $E$ is norm bounded), but the converse is not true in general. For instance, consider the operator $Id_{\mathit{c}_0}: \mathit{c}_0\longrightarrow \mathit{c}_0$. Since the norm of $\mathit{c}_0$ is order
continuous, $Id_{\mathit{c}_0}$ is order L-weakly compact. On the other hand $B_{\mathit{c}_0}$ is not relatively weakly compact, and therefore  $Id_{\mathit{c}_0}$ is not L-weakly compact.
\item It is easy to see that each order L-weakly compact operator is 
order weakly compact (it suffices to note that every L-weakly compact subset of a Banach lattice is relatively weakly compact). The converse is not true in general, for instance, 
consider the operator $T\in L( \ell_1 , \ell_{\infty})$ defined by :
$$\forall (\alpha_n)\in\ell_1,~T((\alpha_n))=(\sum_{n=1}^\infty \alpha_n)(1,1,1,\ldots)$$
Clearly, $T$ is a compact operator (it has rank one), and hence $T$ is order weakly compact.\\
Let $e=(\dfrac{1}{n^2})_{n\in\N^*}$. The sequence $(e_n)$ of the standard unit vectors is a disjoint sequence in the solid hull of $T[0,e]$ $(\vert e_n\vert \ioe T(e) )$. From $\Vert e_n\Vert=1\nrightarrow 0$, we see that $T$ fails to be order L-weakly compact.
\end{itemize}
\end{remark}
\begin{definition}
An operator $T$ from $E$ into $F$ is called order
M-weakly compact if for
every disjoint sequence $(x_n)$ in $B_E$ and every  order bounded  sequence $(f_ n ) $ of $F'$ we have $f_n(T(x_n))\rightarrow 0$
\end{definition}
\begin{remark}
Clearly, every M-weakly compact operator is order M-weakly compact 
(for every sequence $(y_ n ) $ of $F$, if $\|y_n\|\rightarrow 0 $, then for every  order bounded  sequence $(f_ n ) $ of $F'$, $f_n(y_n)\rightarrow 0$), but the converse is not true in general. For instance, consider the operator $Id_{\ell_\infty}$. Since the norm of $\ell_\infty'$ is order
continuous, $Id_{\ell_\infty}$ is order M-weakly compact (see Corollary $\ref{5}$). And since $Id_{\ell_1}$ is not L-weakly compact, $Id_{\ell_\infty}=Id'_{\ell_1}$ is not M-weakly compact.\end{remark}
For our first result, we will need the following lemma.
\begin{lemma}\label{1} \cite[Theorem 5.63]{AB1}\\
For any two nonempty bounded sets $A \subset E$ and
$B \subset E'$ , the following statements are equivalent:
\begin{enumerate}
\item Each disjoint sequence in the solid hull of $A$ converges uniformly
to zero on $B$.
\item Each disjoint sequence in the solid hull of $B$ converges uniformly
to zero on $A$.
\end{enumerate}
\end{lemma}
The following result gives sequential characterizations of order L-weakly compact operators.
\begin{theorem}\label{2}
For an operator $T : E \longrightarrow F$, the following statements are equivalent:
\begin{enumerate}
\item $T$ is order L-weakly compact.
\item For every disjoint sequence $(f_ n )$ of $B_{F'}$ we have $\vert T'(f_n)\vert \rightarrow 0$ for the topology $\sigma(E' , E)$.
\item For every order bounded sequence $(x_ n ) $ of $E$ and every disjoint sequence $(f_ n )$ of $B_{F'}$ we have $f_ n (T( x_ n) ) \rightarrow 0$.
\end{enumerate}
\end{theorem}
\begin{proof}
$(1)\Longleftrightarrow (2)$
$T: E\longrightarrow F$ is order L-weakly compact if and
only if for each $x\in E^+$, every disjoint sequence in the solid hull of $T[0,x]$ converges uniformly to zero on $B_{F'}$. By $Lemma~ \ref{1}$, this is equivalent to
saying that for each $x\in E^+$, every disjoint sequence $(f_n)$ of $B_{F'}$ converges uniformly to zero on $T[0,x]$ (i.e. $\sup\{\vert f_n(T(z))\vert~:~z\in [0,x]\}=\vert T'(f_n)\vert(x)\rightarrow 0$). In other words, $T$ is order L-weakly compact if and only if for every disjoint sequence $(f_ n )$ of $B_{F'}$ we have $\vert T'(f_n)\vert \rightarrow 0$ for the topology $\sigma(E' , E)$.\\
$(2)\Longrightarrow (3)$
Let $(f_n)$ be a disjoint sequence of $B_{F'}$. Let $x\in E^+$, and let $(x_n)$ be a sequence of $[0,x]$. Then it follows from $\vert f_n(T(x_n))\vert\ioe \vert T'(f_n)\vert(x)\longrightarrow 0$ that $f_n(T(x_n))\rightarrow 0$.\\
$(3)\Longrightarrow (2)$ 
Let $(f_n)$ be a disjoint sequence of $B_{F'}$. Assume by way of contradiction that $\vert T'(f_n)\vert \nrightarrow 0$ for the topology $\sigma(E' , E)$.Then, there exists some $x\in E^+$ such that $\vert T'(f_n)\vert(x) \nrightarrow 0$, and consequently there exists some $\varepsilon>0$ and a subsequence $(f_{\varphi(n)} )$ of $(f_ n)$  satisfying $\vert T'(f_{\varphi(n)})\vert(x)=\sup\{\vert f_{\varphi(n)}(T(z))\vert~:~z\in [0,x]\}>\varepsilon$ for all $n\in\N$. Thus for each $n\in\N$, there exists some $x_{\varphi(n)}\in [0,x]$, with $\vert f_{\varphi(n)}(T(x_{\varphi(n)}))\vert >\varepsilon$. From our hypothesis it follows that
$f_{\varphi(n)}(T(x_{\varphi(n)}))\rightarrow 0$, which is absurd. Hence $\vert T'(f_n)\vert \rightarrow 0$ for the topology $\sigma(E' , E)$.
\end{proof}
In a similar way we may prove the following result.
\begin{theorem}\label{3}
An operator $\varphi : E' \longrightarrow F'$ is order L-weakly compact if and only if $\varphi(f_ n) ( y_ n) \rightarrow 0$ for every order bounded sequence $(f_ n ) $ of $E'$ and every disjoint sequence $(y_ n )$ of $B_{F}$.
\end{theorem}
As a consequence, we obtain the following characterizations  of Banach lattices with order continuous norms or whose topological duals have order continuous norms.
\begin{cor}\label{4}
For a Banach lattice $E$ the following statements are equivalent :
\begin{enumerate}
\item $Id_E\in oLW(E)$.
\item The norm of $E$ is order continuous.
\item For every disjoint sequence $(f_ n )$ of $B_{E'}$ we have $\vert f_n \vert \rightarrow 0$ for the topology $\sigma(E' , E)$.
\item $f_ n ( x_ n ) \rightarrow 0$ for every order bounded sequence $(x_ n ) $ of $E$ and every disjoint sequence $(f_ n )$ of $B_{E'}$.
\end{enumerate}
\end{cor}
\begin{cor}\label{5}
For a Banach lattice $E$ the following statements are equivalent :
\begin{enumerate}
\item $Id_{E'}\in oLW(E')$.
\item The norm of $E'$ is order continuous.
\item $f_ n ( x_ n) \rightarrow 0$ for every order bounded sequence $(f_ n ) $ of $E'$ and every disjoint sequence $(x_ n )$ of $B_{E}$.
\item $Id_{E}\in oMW(E)$
\end{enumerate}
\end{cor}
Contrary to weakly compact operators \cite{AB2}, we also deduce that the class of order L-weakly (resp. M-weakly) compact operators satisfies the domination problem.
\begin{cor}\label{6}
Let $S, T : E \longrightarrow F$ be two positive operators such that $0 \ioe S \ioe T$.
Then $S$ is order L-weakly (resp. M-weakly) compact whenever $T$ is one.
\end{cor}
\begin{proposition}\label{7}
Let $E$ and $F$ be two Banach lattices. Then :
\begin{enumerate}
\item The set of all order L-weakly compact operators from $E$ to $F$ is a closed
vector subspace of $L(E, F)$.
\item The set of all order M-weakly compact operators from $E$ to $F$ is a closed
vector subspace of $L(E, F)$.
\end{enumerate}
\end{proposition}
\begin{proof} ${}$
\begin{enumerate}
\item Let $T_1,T_2 \in oLW(E,F)$, and $\alpha\in\R$. Let $(x_n)$ be an order bounded sequence of $E$  and $(f_ n )$ a disjoint sequence of $B_{F'}$.Since $T_1,T_2 \in oLW(E,F)$, it follows from $Theorem ~\ref{2}$, that $$f_n((\alpha T_1+T_2)(x_n))=\alpha f_n(T_1(x_n))+f_n(T_2(x_n))\rightarrow 0$$
Then $\alpha T_1+T_2\in oLW(E,F)$. Thus $oLW(E,F)$ is a vector subspace of $L(E, F)$. To see that it is also a closed vector
subspace of $L(E, F)$, let $T$ be in the closure of $oLW(E,F)$. Let $(x_ n ) $ be an order bounded sequence of $E$  and $(f_ n )$ a disjoint sequence of $B_{F'}$. We have to show that $f_n(T(x_n))\rightarrow 0$. To this end, let $\epsilon >0$. Pick an order L-weakly compact operator $S: E\longrightarrow F$ with $\parallel T-S \parallel<\epsilon$, and note that it
follows from the inequalities :
$$ \begin{array}{l l}
\vert f_n(T(x_n))\vert & \ioe \vert f_n((T-S)(x_n))\vert+\vert f_n(S(x_n))\vert \\
& \ioe \parallel f_n \parallel \parallel T-S \parallel \parallel(x_n)\parallel+\vert f_n(S(x_n))\vert
\end{array}$$
that $\limsup \vert f_n(T(x_n))\vert \ioe \epsilon \parallel(x_n)\parallel$.\\ Since $\epsilon$ is arbitrary, we see that $f_n(T(x_n))\rightarrow 0$ holds, as desired.
\item Clearly, $oMW(E,F)$ is a vector subspace of $L(E,F)$. To see that it is also a closed vector
subspace of $L(E, F)$, let $T$ be in the closure of $oMW(E,F)$. Assume that $(x_ n ) $ is a disjoint sequence of $B_{E}$, and $(f_ n )$ an order bounded sequence of $F'$. We have to show that $f_n(T(x_n))\rightarrow 0$. To this end, let $\epsilon >0$. Pick an order M-weakly compact operator $S: E\longrightarrow F$ with $\parallel T-S \parallel<\epsilon$, and note that it
follows from the inequalities:
$$ \begin{array}{l l}
 \vert f_n(T(x_n))\vert & \ioe \vert f_n((T-S)(x_n))\vert+\vert f_n(S(x_n))\vert \\
& \ioe \parallel f_n \parallel  \parallel T-S \parallel \parallel x_n\parallel+\vert f_n(S(x_n))\vert
\end{array}$$
that $\limsup \vert f_n(T(x_n)) \vert \ioe \epsilon \parallel f_n \parallel \ioe \epsilon \parallel(f_n)\parallel$.\\ Since $\epsilon$ is arbitrary, we see that $f_n(T(x_n))\rightarrow 0$ holds, as desired.
\end{enumerate}
\end{proof}
The classes of L-weakly and M-weakly compact operators are in duality with each other. However, for the classes of order L-weakly and order M-weakly compact operators, we have the following result.
\begin{theorem}\label{8}
Let $E$ and $F$ be two Banach lattices. Then the following statements hold:
\begin{enumerate}
\item An operator $T : E \longrightarrow F$ is order M-weakly compact if and only if its
adjoint $T'$ is order L-weakly compact.
\item For an operator $T: E\longrightarrow F$, if its adjoint $T'$ is order M-weakly compact, then $T$ is order L-weakly compact.
\end{enumerate}
\end{theorem}
\begin{proof} ${}$
\begin{enumerate}
\item Consider an operator $T : E\longrightarrow F$. By $Theorem~\ref{3}$, $T': F'\longrightarrow E'$ is order L-weakly compact if and only if $T'(f_ n) ( x_ n) \rightarrow 0$ for every order bounded sequence $(f_ n ) $ of $F'$ and every disjoint sequence $(x_ n )$ of $B_E$. This is equivalent to saying that $f_ n(T ( x_ n)) \rightarrow 0$ for every order bounded sequence $(f_ n ) $ of $F'$ and every disjoint sequence $(x_ n )$ of $B_E$. In other words, $T'$
is order L-weakly compact if and only if $T$ is order M-weakly compact.
\item Let $T: E\longrightarrow F$ be an operator such that $T'$ is order M-weakly compact.\\
Let $(x_ n ) $ be an order bounded sequence of $E$ and $(f_ n )$ a disjoint sequence of $B_{F'}$. Let $J: E\longrightarrow E''$ be the canonical embedding of $E$ into $E''$.\\
Since $T': F'\longrightarrow E'$ is order M-weakly compact, and  the sequence $(J(x_n))$ of $E''$ is order bounded, $J(x_n)(T'(f_n))=f_n(T(x_n))\rightarrow 0$. Hence $T$ is order L-weakly compact.
\end{enumerate}
\end{proof}
\begin{remark}
However, in general: $T$ is order L-weakly compact $\nRightarrow$ $T'$ is order M-weakly compact.
For instance, the identity operator of $c_0$ is order L-weakly compact (because the norm of $c_0$ is order continuous), but $Id'_{c_0}=Id_{\ell_1}$ is not order M-weakly compact (because the norm of $\ell_1'= \ell_\infty$ is not order continuous ).
\end{remark}
In the following, we give a necessary and sufficient condition for which each order bounded operator is order L-weakly (resp. M-weakly) compact.
\begin{theorem}\label{9}
Let $E$ and $F$ be nonzero Banach lattices. Then the following assertions are equivalent:
\begin{enumerate}
\item Every order bounded operator $T : E\longrightarrow F$ is order L-weakly compact.
\item The norm of $F$ is order continuous.
\end{enumerate}
\end{theorem}
\begin{proof} ${}$
$(1)\Longrightarrow (2)$ Assume by way of contradiction that the norm of $F$ is not order continuous. We have to construct an order bounded operator $T : E \longrightarrow F$ which is not order L-weakly compact. Since the norm of $F$ is not order continuous, it follows from $Theorem ~4.14$ of \cite{AB1} that there exists some $y \in F^ +$ and there exists a disjoint sequence $(y_ n )$ in $[0, y]$ which does not converge to zero in norm. Pick some $a\in E^+$, and $f\in E'^+$ such that $f(a)=1$.
Now, we consider the positive operator $T : E\longrightarrow F$ defined by $T (x) = f (x) y$ for each $x \in E$.\\
$T$ is order bounded but not order L-weakly compact. Indeed, note that $T(a)=y$. So it follows from $(y_n ) \subset [0, y]$ that $(y_ n )$ is a disjoint sequence in the solid hull of $T[0,a]$. Since $(y_n)$ is not norm convergent to zero, $T$ is not order L-weakly compact. But this is in contradiction with our hypothesis $(1)$. Hence, the norm of $F$ is order continuous.\\
$(2)\Longrightarrow (1)$
Follows from the fact that in a Banach lattice with order continuous norm, every order bounded set is L-weakly compact.
\end{proof}
\begin{theorem}\label{10}
Let $E$ and $F$ be nonzero Banach lattices. Then the following assertions are equivalent:
\begin{enumerate}
\item Every order bounded operator $T : E\longrightarrow F$ is order M-weakly compact.
\item The norm of $E'$ is order continuous.
\end{enumerate}
\end{theorem}
\begin{proof} ${}$
$(1)\Longrightarrow (2)$ Assume by way of contradiction that the norm of $E'$ is not order continuous. We have to construct an order bounded operator $T : E \longrightarrow F$ which is not order M-weakly compact. Since the norm of $E'$ is not order continuous, it follows from $Theorem ~4.14$ of \cite{AB1} that there exists some $f \in E'^ +$ and there exists a disjoint sequence $(f_ n )$ in $[0, f]$ which does not converge to zero in norm. Pick some $b\in F^+$, and $g\in F'^+$, such that $g(b)=1$.
Now, we consider the positive operator $T : E\longrightarrow F$ defined by $T (x) = f (x) b$ for each $x \in E$.
$T$ is order bounded, on the other hand, we claim that $T$ is not order M-weakly compact. By $Theorem ~\ref{8}$, it suffices to show that its adjoint $T': F'\longrightarrow E'$ is not order L-weakly compact.
Note that $T' (\varphi) = \varphi (b) f$ for each $\varphi \in F'$.
In particular, $T'(g) = g(b) f = f$. So, $(f_ n )$ is a disjoint sequence in the solid hull of $T' [0,g]$. Since $(f_ n )$ is not norm convergent to zero, $T'$ is not order L-weakly
compact. Hence $T$ is not order M-weakly compact. But this is in contradiction with our hypothesis $(1)$. Hence, the norm of $E'$ is order continuous.\\
$(2)\Longrightarrow (1)$ Let $T:E\longrightarrow F$ be  an order bounded operator. 
By $Theorem ~1.73$ of \cite{AB1} the adjoint operator $T' ~ : F' \longrightarrow E'$ is order bounded.  Since the norm of $E'$ is order continuous, it follows from
$Theorem~\ref{9}$ that $T'$ is order L--weakly compact, and so by $Theorem~\ref{8}$ the operator $T$ is order M-weakly compact.
\end{proof}
\begin{definition}\label{11}
An operator $T$ from
a Banach space $X$ into a Banach lattice $F$ is called strong order bounded if it maps the unit ball of $X$ into an order bounded subset of $F$ .
\end{definition}
\begin{theorem}\label{12}
For an operator $T :E \longrightarrow F$ , the following statements are equivalent :
\begin{enumerate}
\item $T$ is  order L-weakly compact .
\item If $S : X \longrightarrow E$ is a strong order bounded operator from an arbitrary Banach space
$X$ into $E$, the product $T \circ S$ is L-weakly compact.
\item If $S : \ell_1 \longrightarrow E$ is a strong order bounded operator, the product $T \circ S$ is L-weakly compact.
\end{enumerate}
\end{theorem}
\begin{proof} ${}$
\begin{description}
\item[$1\Rightarrow 2 \Rightarrow 3$] Obvious.
\item[$3\Rightarrow 1$] According to $Theorem~\ref{2}$, it suffices to show that $f_n(T (x_ n )) \rightarrow 0$ for every order bounded sequence $(x_ n ) $ of $E$ and every disjoint sequence $(f_ n )$ of $B_{F'}$. Indeed, let $(x_ n )$ and $(f_n)$ be such sequences. Consider the operator $S : \ell_1 \longrightarrow E$ defined by :
$$\forall (\alpha_n)\in\ell_1,~S((\alpha_n))=\sum\limits_{n=1}^{\infty}\alpha_nx_n$$
Now let $(e_n)$ denote the sequence of basic vectors of $\ell_1$. Clearly, $S(e_ n )=x_n$, and $S$ is strong order bounded. So by our hypothesis, the product operator
$T \circ S$ is L-weakly compact. $(e_ n) $ is a norm bounded sequence of $\ell_1$. Then, according to \cite[$Proposition~3.6.2$]{MN1}, $f_n(T\circ S(e_n))\rightarrow 0$. Hence $f_n(T(x_n))=f_n(T\circ S(e_n))\rightarrow 0$, as desired.
\end{description}
\end{proof}
We get, so the following result.
\begin{cor}
For a Banach lattice $E$ the following assertions are equivalent.
\begin{enumerate}
\item The norm of $E$ is order continuous,
\item For each Banach space $X$, every strong order bounded operator $T : X \longrightarrow E$ is
L-weakly compact.
\item Every strong order bounded operator $T : \ell_1\longrightarrow E$ is L-weakly compact.
\end{enumerate}
\end{cor}
\begin{theorem}
For an operator $T :E\longrightarrow F$, the following statements are equivalent :
\begin{enumerate}
\item $T$ is order M-weakly compact .
\item For any operator $S : F \longrightarrow Y$ from $F$  into an arbitrary Banach space $Y$ such that $S'$ is strong order bounded, the product $S \circ T$ is M-weakly compact.
\item For any operator $S : F \longrightarrow \ell_{\infty}$ such that $S'$ is strong order bounded, the product $S \circ T$ is M-weakly compact.\\
\end{enumerate}
\end{theorem}
\begin{proof} ${}$
\begin{description}
\item[$ 1 \Rightarrow 2$]
 Let $T :E\longrightarrow F$ be an order M-weakly compact and $S : F \longrightarrow Y$ such that $S'$ is strong order bounded. It follows from $Theorem~\ref{8}$ that $T'$ is order L-weakly compact, and since $S'$ is strong order bounded, it follows  from $Theorem~\ref{12}$ that $(S\circ T)'=T'\circ S'$ is L-weakly compact. We conclude from $Proposition ~3.6.11.$ of \cite{MN1} that $S\circ T$ is M-weakly compact.
\item[$ 2 \Rightarrow 3$] Obvious.
\item[$3\Rightarrow 1$] Let $(x_n)$ be a disjoint sequence of $B_E$ and let $(f_ n )$ be an order bounded sequence of $F'$. Let $f\in F'_+$ satisfying $\vert f_n\vert \ioe f$ for each $n\in\N$\\
Consider the operator $S : F \longrightarrow \ell_{\infty}$ defined by :
$$\forall y\in F,~S(y)=(f_1(y),f_2(y),\ldots)$$
Since for each $g\in\ell_{\infty}'$, $\vert S'(g)\vert \ioe \Vert g\Vert f$,
it follows that $S'$ is strong order bounded, then by our hypothesis, the product operator
$S \circ T$ is M-weakly compact.\\ Thus $\Vert S \circ T(x_n)\Vert=\Vert (f_1(T(x_n)),f_2(T(x_n)),\ldots)\Vert\rightarrow 0.$ And hence $f_n(T(x_n)) \rightarrow 0$, as desired.
\end{description}
\end{proof}
From $Corollary$ \ref{5}, the norm of $E'$ is order continuous if and only if $Id_{E}$ is order M-weakly compact. We then obtain the following result.
\begin{cor}
For a Banach lattice $E$ the following assertions are equivalent:
\begin{enumerate}
\item The norm of $E'$ is order continuous.
\item For each Banach space $Y$, for every operator $T : E \longrightarrow Y$ such that $T'$ is strong order bounded, $T$ is M-weakly compact.
\item For every operator $T : E \longrightarrow \ell_\infty $, if $T'$ is strong order bounded, then $T$ is M-weakly compact.\end{enumerate}
\end{cor}

\end{document}